\documentclass[12pt]{amsart}
\usepackage{amsmath}
\usepackage{amscd}
\usepackage{amssymb}
\usepackage{amsfonts}

\newtheorem{theorem}{Theorem}[section]
\newtheorem{lemma}[theorem]{Lemma}
\newtheorem{corollary}[theorem]{Corollary}
\newtheorem{proposition}[theorem]{Proposition}

\theoremstyle{definition}

\theoremstyle{remark}
\newtheorem{remark}[theorem]{Remark}
\numberwithin{equation}{section}

\begin{document}
\title[Harmonic and Schr\"odinger functions on shrinkers]
{Harmonic and Schr\"odinger functions of polynomial growth on gradient shrinking Ricci solitons}

\author{Jia-Yong Wu}
\address{Department of Mathematics, Shanghai University, Shanghai 200444, China}
\email{wujiayong@shu.edu.cn}
\author{Peng Wu}
\address{Shanghai Center for Mathematical Sciences, Fudan University, Shanghai 200433, China}
\email{wupenguin@fudan.edu.cn}
\date{\today}
\thanks{}

\begin{abstract}
In this paper, we study harmonic and caloric functions of polynomial growth on
a complete non-compact gradient shrinking Ricci soliton. On one hand, when the
scalar curvature satisfies at least quadratic decay, we prove that the space
of harmonic functions with fixed polynomial growth degree is finite dimensional.
We also prove analogous results for ancient caloric functions. On the other hand,
without any curvature condition, we prove sharp finite dimensional estimates for the
space of Schr\"odinger functions with fixed polynomial growth degree.
\end{abstract}
\maketitle

\section{Introduction}\label{Int1}
For a complete Riemannian manifold $(M,g)$, a smooth function $f$ on $M$ and a real constant
$d\ge 0$, let $\mathcal{H}^f_d(M)$ denote the linear space of $f$-harmonic functions with polynomial
growth of degree at most $d$. That is, $u\in \mathcal{H}^f_d(M)$ if
\[
\Delta_f u=0,
\]
where $f$-Laplacian $\Delta_f:=\Delta-\langle\nabla f,\nabla\rangle$, and for some point
$p\in M$ and a constant $C(u)$ depending on $u$,
\[
\sup_{B_p(r)}|u|\le C(u)(1+r)^d
\]
for sufficiently large $r$, where $B_p(r)$ is the geodesic ball with radius $r$ center
$p$. Clearly, the dimension of $\mathcal{H}^f_d(M)$ is independent of the base point
$p$. When $f$ is constant, that is, in the case of harmonic functions, we simply write
notation $\mathcal{H}_d(M)$ instead of $\mathcal{H}^f_d(M)$.

On an $n$-dimensional manifold $(M,g)$ with nonnegative Ricci curvature, Yau
\cite{[Yau75]} developed a gradient estimate technique to show that any bounded harmonic
function on $M$ must be constant, namely, $\mathrm{dim}\mathcal{H}_0(M)=1$. Moreover,
Yau \cite{[Yau]} conjectured that $\mathrm{dim}\mathcal{H}_d(M)$ is finite for each
$d\ge 0$. For this conjecture, Cheng \cite{[Cheng]} firstly observed that the Cheng-Yau's
gradient estimate \cite{[ChYa]} can be used to show that $\mathrm{dim}\mathcal{H}_d(M)=1$
for $d<1$. For $d=1$, Li and Tam \cite{[LiTa]} proved $\mathrm{dim}\mathcal{H}_1(M)\le n+1$
and the equality is achieved by the Euclidian space $\mathbb{R}^n$. Meanwhile, Cheeger,
Colding and Minicozzi \cite{[CCM]} proved that if $\mathcal{H}_1(M)=n+1$, then $M$ is
isometric to $\mathbb{R}^n$ (see \cite{[Li95]} for the K\"ahler case). Later, Li and Tam
\cite{[LiTa2]}, and Kasue \cite{[Ka]} independently solved the case $n=2$ of Yau's conjecture.
Finally, Colding and Minicozzi \cite{[CoMi]} completely settled Yau's conjecture and proved
that for each $d\ge 1$,
\[
\mathrm{dim}\mathcal{H}_d(M)\le C(n)d^{n-1}
\]
for some constant $C(n)$ depending only on $n$. Here the power of $d$ is sharp, because
$\mathrm{dim}\mathcal{H}_d(\mathbb{R}^n) \sim\frac{2}{(n-1)!}d^{n-1}$ as $d\to \infty$.
For more related results, see \cite{[CLM]}, \cite{[CoMi97]}, \cite{[CoMi98a]}, \cite{[CoMi98b]},
\cite{[HaLi]}, \cite{[Hua]}, \cite{[Hx]}, \cite{[Le]}, \cite{[Li]}, \cite{[Li00]},
\cite{[LW1]}, \cite{[LW2]}, \cite{[Xu]} and references therein.

An $n$-dimensional \emph{gradient shrinking Ricci soliton (or shrinker)} (see \cite{[Ham]})
is a triple $(M, g, f)$ of $n$-dimensional smooth manifold $M$, Riemannian metric $g$ and
a smooth potential function $f$ on $(M,g)$, such that
\begin{align}\label{Eq1}
\mathrm{Ric}+\mathrm{Hess}\,f=\tfrac{1}{2}g,
\end{align}
where $\text{Ric}$ is the Ricci curvature of $(M,g)$ and $\text{Hess}\,f$ is the Hessian
of $f$. The shrinker, generalization of the Einstein manifold, can be
viewed as the self-similar solution to the Ricci flow \cite{[Ham]}. So its geometric
structure is an important subject in the Ricci flow theory.

There exists an interesting phenomenon that non-compact shrinkers often share
same geometric properties of non-compact manifolds with nonnegative Ricci
curvature (even Einstein manifolds), such as the volume growth, Liouville theorem,
$\epsilon$-regularity theorem, Cheeger-Gromov compactness theorem, heat kernel estimate,
splitting result, etc. The reader can consult to \cite{[CaZh]}, \cite{[GeJi]},
\cite{[GeZh]}, \cite{[HaMu]}, \cite{[Hs]}, \cite{[LiWa]}, \cite{[MuWa2]}, \cite{[MuWa3]},
\cite{[MuWa4]}, \cite{[Wu]}, \cite{[WuWu]}, \cite{[WWW]} and reference therein.

Therefore it is natural to ask

\vspace{.1in}

\noindent \textbf{Question}. \emph{For a complete non-compact shrinker $(M, g, f)$,
are $\mathcal{H}^f_d(M)$, $\mathcal{H}_d(M)$ and their related spaces finite
dimensional for each $d\ge 0$?}

\vspace{.1in}

In this paper, we will study this question and an analogous question for ancient
caloric functions on shrinkers. We also discuss similar problems for their
related Schr\"odinger operators. These studies will be useful for understanding
the function theory of shrinkers.

In Corollary 1.6 of \cite{[WuWu]}, we proved that any nonnegative $L_f^1$-integrable
$f$-subharmonic function on shrinker $(M, g, f)$ must be constant. Here, a function
$u$ is called $L_f^1$-integrable, if $\int_M |u|e^{-f}dv<\infty$; $u$ is called
$f$-subharmonic if $\Delta_f u\ge 0$. Suppose $u\in\mathcal{H}^f_d(M)$ for some
$d\ge 0$, then $u^2$ is nonnegative $f$-subharmonic, and $u^2\in L^1_f(M)$ since
$M$ has at most Euclidean volume growth and $f$ grows quadratically (see \cite{[CaZh]}).
Therefore, $u^2\equiv \mathrm{const}$. In other words,
\begin{theorem}\label{main0}
On any non-compact shrinker $(M,g, f)$,
$\mathrm{dim}\mathcal{H}^f_d(M)=1$ for each $d\geq0$.
\end{theorem}

Munteanu and Wang \cite{[MuWa]} proved the space of $f$-harmonic function with 
fixed polynomial growth degree is finite dimensional when $\mathrm{Ric}_f\ge 0$ 
and $f$ is bounded, where $\mathrm{Ric}_f:=\mathrm{Ric}+\mathrm{Hess}\,f$. Ge 
and Zhang \cite{[GeZh]} proved that there does not exist non-constant positive 
$f$-harmonic function on non-compact shrinkers.

Next we focus on the dimension estimate of $\mathcal{H}_d(M)$ on shrinkers. Normalizing
$f$ by a constant in \eqref{Eq1},  we may assume that \eqref{Eq1} satisfies (see the
explanation in Section \ref{sec2} or \cite{[LLW]})
\begin{equation}\label{Eq2}
\mathrm{R}+|\nabla f|^2=f \quad\mathrm{and}\quad \int_M (4\pi)^{-\frac n2}e^{-f} dv=e^{\mu},
\end{equation}
where $\mathrm{R}$ is the scalar curvature of $(M,g)$ and $\mu=\mu(g,1)$ is the entropy
functional of Perelman \cite{[Pe]} (see the definition in Section \ref{sec2}). For the
Ricci flow, the Perelman's functional is time-dependent; but for a shrinker, it is a
fixed constant depending only on the metric $g$.

When the scalar curvature has at least quadratic decay, we prove the following
dimension estimate for $\mathcal{H}_d(M)$ on a shrinker.
\begin{theorem}\label{main1}
Let $(M,g, f)$ be an $n$-dimensional complete non-compact shrinker satisfying \eqref{Eq1}
and \eqref{Eq2}. Assume there exists a nonnegative constant $c_0$ such that the scalar curvature
\[
\mathrm{R}(x)\cdot r^2(x,o)\le c_0,
\]
where $r(x,o)$ is the distance function from point $x\in M$ to a fixed point $o\in M$.
Then the dimension of $\mathcal{H}_d(M)$ is finite for each
$d\ge 0$. Indeed, for each $d\geq0$,
\[
\mathrm{dim}\mathcal{H}_d(M)\le C(n)(c_0+1)^{\frac n2}e^{-\mu}\,4^d
\]
for some constant $C(n)$ depending only on $n$.
\end{theorem}

Our proof relies on a mean value inequality (see Corollary \ref{corelli})
which contains a scalar curvature term, and this enables the scalar curvature
to satisfy some decay condition. Our curvature assumption is suitable to
so-called asymptotically conical shrinkers; see \cite{[MuWa4]}. Recently,
Munteanu, Schulze and Wang \cite{[MSW]} proved that such shrinker has finitely
many ends. On the K\"ahler case, Munteanu and Wang \cite{[MuWa3]} proved
that the space of holomorphic functions with fixed polynomial growth degree
is finite dimensional.

Since the scalar curvature decay condition in Theorem \ref{main1} is specialized,
it is expected to remove this condition. Interestingly, if we consider linear
space $\mathcal{H}_d(a, M)$ instead of $\mathcal{H}_d(M)$, we can prove a finite
dimensional estimate without any assumption. We say that function
$u\in \mathcal{H}_d(a, M)$, if it satisfies the Schr\"odinger equation
\[
(\Delta-a\mathrm{R})u=0,
\]
where $a>0$ is a constant, and for some point $p\in M$ and a constant $C(u)$ depending on $u$,
\[
\sup_{B_p(r)}|u|\le C(u)(1+r)^d
\]
for sufficiently large $r$. The motivation of considering the Schr\"odinger equation comes from
the Perelman's operator $-\Delta+\tfrac14\mathrm{R}$ \cite{[Pe]} and the conformal
Laplacian $-\Delta+\tfrac{n-2}{4(n-1)}\mathrm{R}$ related to the Yamabe problem \cite{[LP]}.
Recently, Li and Wang \cite{[LiWa]} proved a Sobolev inequality containing a scalar curvature
term on shrinkers (see Lemma \ref{lem2}), which also inspires us to consider the
Schr\"odinger operator. For the dimensional estimate of $\mathcal{H}_d(a, M)$, we prove that

\begin{theorem}\label{mainpoten}
Let $(M,g, f)$ be an $n$-dimensional complete non-compact shrinker satisfying \eqref{Eq1}
and \eqref{Eq2}. Then the dimension of $\mathcal{H}_d(a,M)$ is finite for each
$d\ge 0$. Indeed, for each $d\geq0$,
\[
\mathrm{dim}\mathcal{H}_d(a,M)\le C(n)(a^{-1}+1)^{\frac n2}e^{-\mu}\,4^d
\]
for some constant $C(n)$ depending on $n$. For each $d\geq 1$, we could have another
estimate
\[
\mathrm{dim}\mathcal{H}_d(a,M)\le C(n)(a^{-1}+1)^{\frac n2}e^{-\mu}\,d^{n-1}
\]
for some constant $C(n)$ depending on $n$.
\end{theorem}

\begin{remark}\label{example}
The exponent $n-1$ in the second estimate of Theorem \ref{mainpoten} is sharp. Indeed,
on the Gaussian shrinker $(\mathbb{R}^n, g_E, \frac{|x|^2}{4})$, where
$g_E$ is the standard flat Euclidean metric, we know $\mathrm{R}=0$, $\mu=0$
and $\mathrm{dim}\mathcal{H}_d(\mathbb{R}^n)\sim\frac{2}{(n-1)!}d^{n-1}$
as $d\to \infty$.
\end{remark}

\begin{remark}
The results of Theorems \ref{main1} and \ref{mainpoten} can be applied to harmonic and Schr\"odinger
sections on a vector bundle by a similar argument.
\end{remark}

Theorem \ref{mainpoten} does not need any assumption because its proof relies
on another local mean value inequality without any assumption
(see Corollary \ref{corelli2}). From this point, we may think that the
Schr\"odinger operator on shrinkers seems to be more natural than the Laplace
operator in some ways.

Our proof method follows the work Colding-Minicozzi \cite{[CoMi],[CoMi98b]}
and Li's generalization \cite{[Li]}. Inspired by their arguments, there are
two main technical ingredients in our setting. One is a local mean value
inequality (see Corollaries \ref{corelli} and \ref{corelli2}). Its proof
depends on the Moser iteration to the Li-Wang's Sobolev inequality \cite{[LiWa]}.
The other (for the second estimate of Theorem \ref{mainpoten}) is an integral
inequality involving a class of inner products on a subspace of
$\mathcal{H}_d(a, M)$ (see Lemma \ref{parallemm2}). To prove the integral
inequality, we use an area comparison (see Remark \ref{areacomp}) instead of
a volume comparison considered by Li \cite{[Li]} on manifolds. At present,
we do not know if there exists a volume comparison on shrinkers.

There have been various Liouville type results
for $f$-harmonic functions under different conditions. The reader can refer
to Brighton \cite{[Br]}, Cao and Zhou \cite{[CaZh]}, Ge and Zhang \cite{[GeZh]},
Hua, Liu and Xia \cite{[HLX]}, Munteanu and Wang \cite{[MuWa], [MuWa1]}, Petersen
and Wylie \cite{[PeWy],[PeWy2]}, Pigola, Rimoldi and Setti \cite{[PiRS]}, Wei
and Wylie \cite{[WeWy]}, Wu \cite{[Wu]}, Wu and Wu \cite{[WuWu],[WuWu2]} and
references therein.

A natural generalization of ($f$-)harmonic function is an ancient ($f$-)caloric
function on shrinker $(M,g,f)$. Here, ancient $f$-caloric functions mean that
smooth solutions $u(x,t)$, defined on $M\times(-\infty,0]$, to the
$f$-heat equation
\[
(\Delta_f-\partial_t)u=0.
\]
Similar to harmonic function cases, when the scalar curvature has at least
quadratic decay, we show that the space of ancient caloric functions with
fixed polynomial growth degree is finite dimensional (see Theorem \ref{Main1c}).
We also show that the space of $f$-ancient caloric functions with fixed
polynomial growth degree on is finite dimensional. In particular, when the
degree is less than one, such space is one-dimensional (see Theorem \ref{main2}).
The proof of caloric results mainly depends on a recent Colding-Minicozzi's
result \cite{[CoMi2]} and the gradient estimate technique \cite{[SZ],[Wu]}.
For these results, see Section \ref{sec6} for details.

The rest of the paper is organized as follows. In Section \ref{sec2}, we recall some
properties of shrinkers. In particular, we give a Laplacian comparison and a new
volume comparison. In Section \ref{sec3}, we prove local mean value inequalities
for harmonic functions and Schr\"odinger equations on shrinkers. In Section
\ref{sec4}, adapting Li's proof idea in \cite{[Li]}, we apply the mean value
inequality for harmonic functions to prove Theorem \ref{main1}. In Section
\ref{sec5}, we apply the mean value inequality for Schr\"odinger functions to
prove Theorem \ref{mainpoten}. In Section \ref{sec6}, we prove that ($f$-)caloric
function spaces are finite dimensional on shrinkers.

\textbf{Acknowledgement}.
The authors thank Ruixuan Li for helpful discussions. The first author thanks the hospitality
of Shanghai Center for Mathematical Sciences, where part of the work was done. The first author
was partially supported by NSFS (17ZR1412800) and NSFC (11671141). The second author was partially
supported by NSFC (11701093).

\section{Preliminaries}\label{sec2}
In this section, we collect some basic facts about shrinkers, which will be
repeatedly used in the proof of theorems. In particular, we give a weighted
Laplacian comparison and a new volume comparison. For more results, see
\cite{[Caoh]}, \cite{[Chowetc]} and references therein for nice surveys.

In the rest of the whole paper, we let $C(n)$ denote a constant depending only
on dimension $n$ of shrinker $(M,g,f)$ whose value may change from line to line.

From shrinker \eqref{Eq1}, Hamilton observed that $\mathrm{R}+\Delta f=\frac n2$ and
\begin{equation}\label{condition}
C(g):=\mathrm{R}+|\nabla f|^2-(f+c)
\end{equation}
is a finite constant, where $c\in \mathbb{R}$ is a free parameter.
Ordinarily, $c$ is chosen to be zero in many literatures, e.g. \cite{[CaZh]}, \cite{[Chowetc]}.
Combining the above equations, we get
\begin{equation}\label{identitycond}
2\Delta f-|\nabla f|^2+\mathrm{R}+(f+c)-n=-C(g).
\end{equation}

By Chen's work (see Proposition 2.2 in \cite{[Chen]}), $\mathrm{R}\ge 0$.
Moreover, by Theorem 3 in \cite{[PiRS]}, $\mathrm{R}>0$ unless $(M,g,f)$
is the Gaussian shrinker $(\mathbb{R}^n, g_E, \frac{|x|^2}{4})$. By Cao-Zhou
\cite{[CaZh]}, $f$ is uniformly equivalent to the distance function
squared. Precisely, we have the following sharp estimate originally due to
Cao-Zhou \cite{[CaZh]} and later improved by Haslhofer-M\"uller \cite{[HaMu]}.

\begin{lemma}\label{potenesti}
Let $(M,g, f)$ be an $n$-dimensional complete non-compact shrinker
satisfying \eqref{Eq1} and \eqref{condition}. Then there exists a point $p_0\in M $ where
$f$ attains its infimum (may be not unique). Moreover, $f$ satisfies
\[
\tfrac 14\left[\left(r(x,p_0)-5n\right)_{+}\right]^2\le f(x)+c+C(g)\le\tfrac 14\left(r(x,p_0)+\sqrt{2n}\right)^2,
\]
where $r(x,p_0)$ is a distance function from $p_0$
to $x$, and $a_+=\max\{a,0\}$ for $a\in \mathbb{R}$.
\end{lemma}
\begin{remark}\label{potenest}
On a shrinker with \eqref{Eq1} and \eqref{Eq2}, for any point $p\in M$
(maybe not a infimun point of $f$), by Theorem 1.1 of Cao-Zhou \cite{[CaZh]}
and later refined by Chow et al. \cite{[Chowetc]},
\[
\tfrac 14\left[(r(x,p)-2\sqrt{f(p)}-4n+\tfrac 43)_{+}\right]^2
\le f(x)\le\tfrac 14\left(r(x,p)+2\sqrt{f(p)}\right)^2
\]
for any $x\in M$.
\end{remark}

Lemma \ref{potenesti} implies $\mathrm{R}(x)\le\tfrac 14(r(x,p_0)+\sqrt{2n})^2$.
However, it remains an interesting question if $\mathrm{R}(x)$ is bounded from
above by a constant.

By a Cao-Zhou's result \cite{[CaZh]}, and the Munteanu's improvement \cite{[Mun]}
(see also Haslhofer-M\"uller \cite{[HaMu]}), we have
\begin{lemma}\label{volesti}
Let $(M,g, f)$ be an $n$-dimensional complete non-compact shrinker
with $p_0\in M$ as in Lemma \ref{potenesti}. For any $r\ge 0$,
\[
V_{p_0}(r)\le C(n)r^n,
\]
where $V_{p_0}(r)$ denotes the volume of the geodesic ball $B_{p_0}(r)$.
\end{lemma}
Lemma \ref{volesti} gives a volume growth at any point. This can be regarded
as an analog of the Bishop's theorem for manifolds with nonnegative Ricci curvature.

\begin{proposition}\label{volesti2}
Let $(M,g, f)$ be an $n$-dimensional complete non-compact shrinker with $p_0\in M$
as in Lemma \ref{potenesti}. For any point $p\in M$,
\[
V_p(r)\le C(n)r^n
\]
for all $r\ge r(p,p_0)$.
\end{proposition}
\begin{proof}[Proof of Proposition \ref{volesti2}]
For any point $p\in M$ and $r>0$,
\begin{equation}\label{inset}
B_p(r)\subset B_{p_0}(r+r(p,p_0)).
\end{equation}
By Lemma \ref{volesti}, there exists a constant $C(n)$ such that
$V_{p_0}(r+r(p,p_0))\leq C(n)(r+r(p,p_0))^n$ for all $r$. Letting
$r\ge r(p,p_0)$, then $V_{p_0}(r+r(p,p_0))\leq 2^nC(n)r^n$.
Combining this with \eqref{inset} gives the conclusion.
\end{proof}

\begin{remark}\label{areacomp}
In geodesic polar coordinates of $(M,g,f)$, let $\mathcal{A}(r,\theta)$ be the
volume element of $g$ and let $A_p(r):=\int_{S^{n-1}}\mathcal{A}(t,\theta)d\theta^{n-1}$
be the volume of geodesic sphere $S(p,r)=\{x\in M|d(x,p)=r\}$.
Since $V_p(r)=\int^r_0A(p,t)dt$, by Proposition \ref{volesti2},
\begin{equation}\label{areaup}
A_p(r)\le C(n)r^{n-1}
\end{equation}
for $r\ge r(p,p_0)$. This estimate will be used in the proof of Lemma
\ref{parallemm2}.
\end{remark}

Recall that on an $n$-dimensional complete manifold $(M,g)$,
the Perelman's $\mathcal{W}$-entropy functional \cite{[Pe]} is defined as
\[
\mathcal{W}(g,\widetilde{f},\tau)
:=\int_M\Big[\tau\Big(|\nabla \widetilde{f}|^2+\mathrm{R}\Big)+\widetilde{f}-n\Big](4\pi\tau)^{-n/2}e^{-\widetilde{f}}dv
\]
for some $\widetilde{f}\in C^\infty(M)$ and $\tau>0$ if it is finite,
and the Perelman's $\mu$-entropy functional \cite{[Pe]} is defined as
\[
\mu(g,\tau):=\inf\Big\{\mathcal{W}(g,\widetilde{f},\tau)\Big|\widetilde{f}\in C_0^\infty(M)\,\,\,\mathrm{with}\,\,\, \int_M(4\pi\tau)^{-n/2}e^{-\widetilde{f}}dv=1\Big\}.
\]
In general, the minimizer of $\mu(g,\tau)$ may not exist on
a noncompact manifold. However, by the Carrillo-Ni's argument \cite{[CaNi]}
(see also \cite{[HaMu]}), on complete (possible non-compact) shrinker,
$f+c$ is always a minimizer of $\mu(g,1)$. Therefore,
\begin{equation*}
\begin{aligned}
\mu(g,1)=\mathcal{W}(g,f+c,1)
&:=\int_M\Big(|\nabla f|^2+\mathrm{R}+(f+c)-n\Big)(4\pi)^{-n/2}e^{-(f+c)}dv\\
&=\int_M\Big(2\Delta f-|\nabla f|^2+\mathrm{R}+(f+c)-n\Big)(4\pi)^{-n/2}e^{-(f+c)}dv\\
&=-C(g),
\end{aligned}
\end{equation*}
where $c$ is a constant such that $\int_M(4\pi)^{-n/2}e^{-(f+c)}dv=1$.
In the last equality, we used \eqref{identitycond}. Notice that the above
integral formulas always hold due to Lemma \ref{potenesti}. Hence
\[
\mathrm{R}+|\nabla f|^2-(f+c)=-\mu(g,1)
\]
with $\int_M(4\pi)^{-n/2}e^{-(f+c)}dv=1$.
If we let $c=\mu(g,1)$, then we get \eqref{Eq2} in introduction.

On a shrinker, there exists the following weighted
Laplacian comparison, which will be used in Section \ref{sec6}.

\begin{proposition}\label{meancom}
Let $(M,g, f)$ be an $n$-dimensional complete non-compact shrinker satisfying
\eqref{Eq1} and \eqref{Eq2}. Let $\gamma$ be a minimizing normal geodesic in
$M$ with $p=\gamma (0)$ and $x=\gamma (r)$. Then
\[
\Delta_f\,r(p,x)\le-\tfrac r2+3\sqrt{f(p)}+4n-\tfrac 43
\]
for all $r\ge 2\sqrt{f(p)}+4n-\tfrac 43$. In particular, $\Delta_f\,r(p,x)\le 0$
for all $r\ge 6\sqrt{f(p)}+8n-\tfrac 83$.
\end{proposition}
\begin{proof}[Proof of Proposition \ref{meancom}]
We adapt Wei-Wylie's argument \cite{[WeWy]} to prove the result. For any point
$p\in M$, let $\gamma$ be the minimizing normal geodesic from $p$ to $x\in M$ such
that $\gamma(0)=p$ and $\gamma (r)=x$. Let $r(x):=d(x,p)$ be a distance function from
$p$ to $x$. Applying the Bochner formula to $r(x)$ and using the fact $|\nabla r|=1$,
\[
0=\tfrac 12\Delta|\nabla r|^2=|\text{Hess}\,r|^2+\tfrac{\partial}{\partial r}(\Delta\,r)
+\mathrm{Ric}\left(\nabla r,\nabla r\right)
\]
outside the cut locus of $p$. Here, $\text{Hess}\,r$ is the second fundamental
form of the geodesic sphere $S_p(r)$ and $\Delta\,r$ is the mean curvature
$m(r)$ of $S_p(r)$ in the outer normal direction. Let $\lambda_1,\ldots,\lambda_n$
be the eigenvalues of $\mathrm{Hess}\,r$. Without loss of generality we assume
$\lambda_1=0$, because the exponential function is a radial isometry. By the
Cauchy-Schwarz inequality,
\[
|\mathrm{Hess}\,r|^2=\lambda^2_2+\ldots+\lambda^2_n\ge \frac{(\lambda_2+\ldots+\lambda_n)^2}{n-1}=\frac{[\mathrm{tr}(\mathrm{Hess}\,r)]^2}{n-1}=\frac{(\Delta r)^2}{n-1}.
\]
Therefore, we obtain a Riccati inequality
\[
m'(r)+\frac{1}{n-1}m^2(r)+\mathrm{Ric}\left(\nabla r,\nabla r\right)\le 0
\]
along the minimizing normal geodesic $\gamma $. As in \cite{[MuWa2]},
multiplying the above inequality by $r^2$ and integrating from $0$ to $t>0$,
we have
\begin{equation}\label{boch2}
\int_0^tm^{\prime }(r) r^{2}dr+\frac{1}{n-1}%
\int_0^tm^{2}( r ) r^{2}dr+\frac 12
\int_0^tr^{2}dr\leq \int_{0}^{t}f^{\prime \prime }(r)
r^{2}dr,
\end{equation}
where we used \eqref{Eq1} and $f^{\prime\prime}(r):=\text{Hess}\,f\left(\nabla r,\nabla r\right)
=\frac{d^2}{dr^2}(f\circ\gamma)(r)$. Integrating the first and the last terms in \eqref{boch2}
by parts and rearranging terms,
\[
m(t)t^2+\frac{1}{n-1}\int_0^t\Big(m(r)r-(n-1)\Big)^2dr\le(n-1)t-\frac{t^3}{6}
+t^2f'(t)-2\int_0^tf'(r)rdr.
\]
Discarding the non-negative second term above, we get
\begin{equation}\label{keineq}
m_f(t) \le \frac{n-1}{t}-\frac{t}{6}-\frac{2}{t^{2}}\int_0^tsf^{\prime}(s) ds,
\end{equation}
where $m_f(t):=m(t)-f'(t)$. Integrating the last term by parts yields
\begin{equation}\label{boch3}
m_f(t) \leq \frac{n-1}{t}-\frac{t}{6}-\frac{2}{t}f(t)+\frac{2}{t^{2}}\int_{0}^{t}f(s)ds.
\end{equation}
By Remark \ref{potenest},
\[
\tfrac 14\left[(r(x,p)-2\sqrt{f(p)}-4n+\tfrac 43)_{+}\right]^2
\le f(x)\le \tfrac 14\left(r(x,p)+2\sqrt{f(p)}\right)^2
\]
for any $x\in M$. Substituting this into  \eqref{boch3} yields
\[
m_f(r)\le\frac{n-1}{r}-\frac r6-\frac{1}{2r}\left[\left(r-2\sqrt{f(p)}-4n+\frac 43\right)_{+}\right]^2
+\frac{1}{2r^2}\int_0^r\left(s+2\sqrt{f(p)}\right)^2 ds
\]
for any $r>0$. Now we choose $r\geq 2\sqrt{f(p)}+4n-\frac 43$, and integrate the last term
in the above inequality. We finally get that
\[
m_f(r)\le-\tfrac r2+3\sqrt{f(p)}+4n-\tfrac 43
\]
for all $r\geq 2\sqrt{f(p)}+4n-\frac 43$. Then the result follows by $m_f(r)=\Delta_f\,r$.
\end{proof}

On a shrinker, Cao and Zhou \cite{[CaZh]} proved an interesting weak volume
comparison for sub-level balls. We observe that there exists a new volume
comparison for geodesic balls.
\begin{proposition}\label{compari}
Let $(M,g, f)$ be an $n$-dimensional complete non-compact shrinker satisfying
\eqref{Eq1} with $p_0\in M$ as in Lemma \ref{potenesti}. For any point $p\in M$,
\[
V_p(R_2)-V_p(R_1)\le \tfrac{1}{n}e^{\frac n2}\big(1.1 {R_2}^n-0.9{R_1}^n\big)
\]
for any $R_2\ge R_1\ge(\sqrt[n]{1.1}-1)^{-1}\,r(p,p_0)$.
\end{proposition}
\begin{remark}
Constants 1.1 and 0.9 are not unique. Indeed, the number 1.1 can be replaced by
any slightly bigger number than 1, and the number 0.9 can be replaced by any
slightly smaller number than 1; at this time $R_2\ge R_1$ are chosen much larger
than the previous case.
\end{remark}
\begin{proof}[Proof of Proposition \ref{compari}]
For a infimum point $p_0\in M$ of $f$, in geodesic polar coordinates, the volume element
is written as $dv=\mathcal{A}(p_0,r,\theta)dr\wedge d\theta^{n-1}$,
where $d\theta^{n-1}$ is the volume element of the unit sphere $S^{n-1}$ in $T_{p_0}M$.
Following the proof trick of Theorem 2.2 in \cite{[MuWa2]}, from \eqref{keineq},
by using Lemma \ref{potenesti} and $f(p_0)\le \frac n2$, we finally get
$\mathcal{A}(p_0,r)\le e^{f(p_0)}r^{n-1}\le e^{\frac n2}r^{n-1}$ for all $r>0$.
Integrating this from $r_1$ to $r_2(\ge r_1)$ with respect to the $r$-variable,
\begin{equation}\label{difference}
V_{p_0}(r_2)-V_{p_0}(r_1)\le\tfrac{1}{n}e^{\frac n2}(r^n_2-r^n_1).
\end{equation}

On the other hand, for any point $p\in M$ and $R_2\ge R_1\ge r(p,p_0)$,
$V_p(R_2)\leq V_{p_0}(R_2+r(p,p_0))$ and $V_p(R_1)\geq V_{p_0}(R_1-r(p,p_0))$.
Therefore,
\begin{equation}
\begin{aligned}\label{voldiff}
V_p(R_2)-V_p(R_1)&\le V_{p_0}(R_2+r(p,p_0))-V_{p_0}(R_1-r(p,p_0))\\
&\le \tfrac{1}{n}e^{\frac n2}\left[\big(R_2+r(p,p_0)\big)^n-\big(R_1-r(p,p_0)\big)^n\right],
\end{aligned}
\end{equation}
where we used \eqref{difference} in the second inequality above. Now we choose $R_2$
and $R_1$ sufficiently large such that $R_2\ge R_1\ge(\sqrt[n]{1.1}-1)^{-1}\,r(p,p_0)$,
which implies that
\[
R_2+r(p,p_0)\le \sqrt[n]{1.1} R_2 \quad\mathrm{and}\quad  R_1-r(p,p_0)\ge(2-\sqrt[n]{1.1})R_1\ge \sqrt[n]{0.9}R_1.
\]
Substituting two above estimates into \eqref{voldiff} proves the result.
\end{proof}

\section{Mean value inequality}\label{sec3}
Recently, Li and Wang \cite{[LiWa]} applied the Perelman's entropy functional
and the Markov semigroup technique of Davies \cite{[Dav]} to obtain a local
Sobolev inequality on shrinkers.
\begin{lemma}[Li-Wang \cite{[LiWa]}]\label{lem2}
Let $(M,g, f)$ be an $n$-dimensional shrinker satisfying \eqref{Eq1}
and \eqref{Eq2}. Then for each $u\in C^{\infty}_0(B_p(r))$, where 
$p\in M$ and $r>0$,
\begin{equation}\label{sobo}
\left(\int_{B_p(r)} u^{\frac{2n}{n-2}}\,dv\right)^{\frac{n-2}{n}} \le C(n) e^{-\frac{2\mu}{n}} \int_{B_p(r)}\left(4|\nabla u|^2+\mathrm{R}\,u^2\right) dv.
\end{equation}
\end{lemma}
The appearance of Sobolev inequality \eqref{sobo} is similar to the classical
Sobolev inequality on compact manifolds. By Li-Li-Wang's work (see Lemma 2.5
in \cite{[LLW]}), we have
\[
(4\pi)^{\frac n2}e^{-2^{4n+7}}\cdot e^{\mu}
\le V_{p_0}(1)\le(4\pi)^{\frac n2}e^n\cdot e^{\mu}
\]
on shrinker $(M,g, f)$, which means that $e^{\mu}$ is almost equivalent
to the volume of unit ball $B_{p_0}(1)$. Here $p_0\in M$ is a infimum point of $f$.

Using Lemma \ref{lem2}, we get a local mean value inequality for the
heat equation by the standard Moser iteration, which is a key
step to prove Theorem \ref{main1}.

\begin{theorem}\label{prop}
Let $(M,g, f)$ be an $n$-dimensional shrinker
satisfying \eqref{Eq1} and \eqref{Eq2}.  Assume there exists a nonnegative constant
$c_0$ such that
\[
\mathrm{R}(x)\cdot r^2(x,o)\le c_0.
\]
Fix $0<m<\infty$. There exists a
positive constant $C(n,m)$ depending on $n$ and $m$ such that for any
$s\in \mathbb{R}$, for any $0<\delta<1$, and for any smooth nonnegative solution $u$ of
\[
\left(\Delta-\partial_t\right)u(x,t)\ge0
\]
in the space-time cylinder $Q:=B_p(r)\times(s-r^2,s)$, where $p\in \partial B_o(2r)$, we have
\begin{equation}\label{Lm-mean}
\sup_{Q_\delta}\{u^m\}
\leq \frac{C(n,m)(c_0+1)^{\frac n2}}{(1-\delta)^{2+n}\,e^{\mu}\,r^{2+n}}\int_Qu^m\,\,\, dv\, dt,
\end{equation}
where $Q_\delta:=B_p(\delta r)\times (s-\delta r^2,s)$
and $\mu:=\mu(g,1)$ is the Perelman's entropy functional.
\end{theorem}

\begin{proof}[Proof of Theorem \ref{prop}]
The proof is analogous to the argument of Proposition 2.6 in \cite{[WuWu]}, where
main trick is the Moser iteration. We need to carefully examine the explicit 
coefficients of the mean value inequality in terms of the Sobolev constant in 
\eqref{sobo}.

We first confirm the case $m=2$ of \eqref{Lm-mean}. Case $m>2$ then follows by the case $m=2$.
Indeed, we let
$v=u^m$, where $m\geq1$. Then
\[
\left(\Delta-\partial_t\right) u^m=mu^{m-1}\Delta u+m(m-1)u^{m-2}|\nabla u|^2-mu^{m-1}u_t
\ge mu^{m-1}\left(\Delta-\partial_t\right) u.
\]
This means that if $u$ is a nonnegative solution of $(\Delta-\partial_t)u\ge 0$,
then $v$ is also a nonnegative solution of $(\Delta-\partial_t)v\ge 0$.

For any nonnegative function $\phi\in C^{\infty}_0(B)$, where $B=B_p(r)$,
multiplying $(\Delta-\partial_t)u\ge 0$ by $\phi$ and integrating it over
$B$,
\[
\int_B(\phi u_t+\nabla\phi\nabla u)dv\le 0.
\]
Letting $\phi=\psi^2u$, where $\psi\in C^{\infty}_0(B)$, then
\[
\int_B(\psi^2 uu_t+\psi^2|\nabla u|^2)dv\le 2\left|\int_B u\psi \nabla u\nabla\psi dv\right|
\le 2\int_B|\nabla\psi|^2u^2dv+\frac 12\int_B\psi^2|\nabla u|^2dv,
\]
which implies that
\[
\int_B(2\psi^2 uu_t+|\nabla(\psi u)|^2)dv\le 4\,\|\nabla\psi\|^2_{\infty} \int_{\mathrm{supp}(\psi)} u^2dv.
\]

Multiplying both sides of the above inequality by a smooth time function $\lambda:=\lambda(t)$,
which will be determined later, and integrating by parts, we get
\begin{equation}\label{basinequ}
\partial_t\left(\int_B(\lambda\psi u)^2dv\right)+\lambda^2\int_B|\nabla(\psi u)|^2dv
\le C\lambda\left(\lambda\|\nabla \psi\|^2_{\infty}+|\lambda'|\sup\psi^2\right)\int_{\mathrm{supp}(\psi)} u^2dv,
\end{equation}
where $C$ is finitely constant which may change from line to line below.

Choose $\psi$ and $\lambda$ satisfying the following two properties:
\begin{enumerate}
\item
$0\leq\psi\leq 1$, $\mathrm{supp}(\psi)\subset\sigma B$, \, $\psi=1$ in $\sigma' B$ and
$|\nabla\psi|\leq \frac{2}{\kappa\,r}$;
\item
$0\leq\lambda\leq 1$, $\lambda=0$ in $(-\infty,s-\sigma r^2)$,  $\lambda=1$ in $(s-\sigma' r^2,+\infty)$, and
$|\lambda'(t)|\leq \frac{2}{\kappa^2 r^2}$,
where $0<\sigma'<\sigma<1$, $\kappa=\sigma-\sigma'$.
\end{enumerate}

Setting $I_\sigma=(s-\sigma r^2,s)$ and integrating \eqref{basinequ} over the interval
$(s-r^2, t)$ with $t\in I_{\sigma'}$,
\begin{equation}
\begin{aligned}\label{integso}
\sup_{I_{\sigma'}}\left\{\int_B\psi u^2dv\right\}+\int\int_{B\times I_{\sigma'}}|\nabla(\psi u)|^2dv dt
\leq \frac{C}{\kappa^2 r^2}\int\int_{Q_\sigma} u^2dv dt.
\end{aligned}
\end{equation}

On the other hand, using the H\"older inequality
\[
\int_{\widetilde{B}}\varphi^{2(1+\frac 2n)}dv
\le\left(\int_{\widetilde{B}}|\varphi|^{\frac{2n}{n-2}}dv\right)^{\frac{n-2}{n}}
\left(\int_{\widetilde{B}}\varphi^2dv\right)^{\frac{2}{n}}
\]
for any geodesic ball $\widetilde{B}$, and Lemma \ref{lem2}, we have that
\[
\int_{\widetilde{B}}\varphi^{2(1+\frac 2n)}dv
\le \left(\int_{\widetilde{B}}\varphi^2dv\right)^{\frac{2}{n}}
\left[C(n)e^{-\frac{2\mu}{n}}\int_{\widetilde{B}}(4|\nabla \varphi|^2+\mathrm{R}\,\varphi^2)dv\right]
\]
for any $\varphi\in C^{\infty}_0(\widetilde{B})$. Set $\varphi=u$ and $\widetilde{B}=B_p(\sigma' r)$
in above. Then we integrate it from $s-\sigma'r^2$ to $s$ with respect to
the time variable and get that
\begin{equation}
\begin{aligned}\label{space-timeineq}
&\int^s_{s-\sigma'r^2}\int_{B_p(\sigma' r)}u^{2(1+\frac 2n)}dvdt\\
&\quad\le C(n)e^{-\frac{2\mu}{n}}\left(\int_{B_p(\sigma' r)}u^2dv\right)^{\frac{2}{n}}
\left[\int^s_{s-\sigma'r^2}\int_{B_p(\sigma' r)}(4|\nabla u|^2+\mathrm{R}\,u^2)dvdt\right].
\end{aligned}
\end{equation}
In the following, we shall apply \eqref{integso} three times to carefully
estimate the right hand side of \eqref{space-timeineq}. Firstly, we observe
\begin{equation}\label{inequality1}
\int_{B_p(\sigma' r)}u^2dv\le\int_{B_p(r)}\psi u^2dv
\le\frac{C}{\kappa^2 r^2}\int\int_{Q_\sigma} u^2dv dt,
\end{equation}
where we used $\psi\equiv 1$ in $B_p(\sigma' r)$, $\sigma'<1$, in the first inequality
and used \eqref{integso} in the second inequality. Secondly, we similarly have
\begin{equation}\label{inequality2}
\int^s_{s-\sigma'r^2}\int_{B_p(\sigma' r)}|\nabla u|^2dvdt
\le\int^s_{s-\sigma'r^2}\int_{B_p(r)}|\nabla\psi u|^2dvdt
\le \frac{C}{\kappa^2 r^2}\int\int_{Q_\sigma} u^2dv dt.
\end{equation}
Thirdly, since $\mathrm{R}(x)\le \frac{c_0}{r^2(x,o)}\le\frac{c_0}{r^2}$
for $p\in \partial B_o(2r)$ and $x\in B_p(r)$, then
\begin{equation}
\begin{aligned}\label{inequality3}
\int^s_{s-\sigma'r^2}\int_{B_p(\sigma' r)}\mathrm{R}\,u^2dvdt
&\le c_0\cdot\sup_{I_{\sigma'}}\left\{\int_{B_p(\sigma' r)}u^2dv\right\}\\
&\le c_0\cdot\sup_{I_{\sigma'}}\left\{\int_{B_p(r)}\psi u^2dv\right\}\\
&\le c_0\cdot \frac{C}{\kappa^2 r^2}\int\int_{Q_\sigma} u^2dv dt,
\end{aligned}
\end{equation}
where we still used the fact that $\psi\equiv 1$ in $B_p(\sigma' r)$, $\sigma'<1$, and
\eqref{integso}.
Substituting \eqref{inequality1}, \eqref{inequality2} and \eqref{inequality3} into
\eqref{space-timeineq} yields
\begin{equation}\label{Moserinequ}
\int\int_{Q_{\sigma'}}u^{2\theta}dv dt
\leq E(B)\left(\frac{C}{\kappa^2 r^2}\int\int_{Q_\sigma} u^2dv dt\right)^\theta
\end{equation}
with $\theta=1+2/n$, where $E(B):=C(n)e^{-\frac{2\mu}{n}}(c_0+1)$.

Now for any $m\geq 1$, $u^m$ is also a nonnegative solution of
$(\Delta-\partial_t)v\ge 0$. This implies that
\begin{equation}\label{intds2}
\int\int_{Q_{\sigma'}}u^{2m\theta}dv dt
\leq E(B)\left(\frac{C}{\kappa^2 r^2}\int\int_{Q_\sigma} u^{2m}dv dt\right)^\theta
\end{equation}
for $m\geq1$.

\vspace{.1in}

Let $\kappa_i=(1-\delta)2^{-i}$, which satisfies $\Sigma^{\infty}_1\kappa_i=1-\delta$.
Let $\sigma_0=1$, $\sigma_{i+1}=\sigma_i-\kappa_i=1-\Sigma^i_1\kappa_j$. Applying \eqref{intds2}
for $m=\theta^i$, $\sigma=\sigma_i$, $\sigma'=\sigma_{i+1}$, we have
\[
\int\int_{Q_{\sigma_{i+1}}}u^{2\theta^{i+1}}dv dt
\leq E(B)\left\{C^{i+1}\left[\Big(1-\delta\Big)r\right]^{-2}\int\int_{Q_{\sigma_i}} u^{2\theta^i}dv dt\right\}^\theta.
\]
Therefore
\[
\left(\int\int_{Q_{\sigma_{i+1}}}u^{2\theta^{i+1}}dv dt\right)^{\theta^{-i-1}}
\le C^{\Sigma j\theta^{1-j}}\cdot E(B)^{\Sigma\theta^{-j}}\cdot\left[\Big(1-\delta\Big)r\right]^{-2\Sigma\theta^{1-j}}\int\int_Q u^2dv dt,
\]
where $\Sigma$ denotes the summations from $1$ to $i+1$. Letting $i\to \infty$,
\begin{equation}\label{prmi}
\sup_{Q_\delta}\{u^2\}\leq C\cdot E(B)^{\frac n2}\cdot[(1-\delta)r]^{-2-n}\|u\|^2_{2,Q}.
\end{equation}
This is regarded as a $L^2$-mean value type inequality on a complete shrinker $(M,g, f)$.
By the preceding explanation, we hence prove \eqref{Lm-mean} when $m\geq 2$.

When $0<m<2$, we are still able to obtain \eqref{Lm-mean}. It is proved from \eqref{prmi}
by a different iterative argument. Letting $\sigma\in (0,1)$ and $\rho=\sigma+(1-\sigma)/4$,
then \eqref{prmi} indeed implies
\[
\sup_{Q_\sigma}\{u\}\leq F(B)\cdot(1-\sigma)^{-1-\frac n2}\|u\|_{2,Q_{\rho}},
\]
where $F(B):=C(n)e^{-\frac \mu2}(c_0+1)^{\frac n4}r^{-1-\frac n2}$.
Using inequality $\|u\|_{2,Q}\leq \|u\|^{1-\frac m2}_{\infty,Q}\cdot\|u\|^{\frac m2}_{m,Q}$,
$0<m<2$, for any $Q$, we obtain
\begin{equation}\label{moserit}
\|u\|_{\infty,Q_\sigma} \leq F(B)\|u\|^{\frac m2}_{m,Q}\cdot(1-\sigma)^{-1-\frac n2}\|u\|^{1-\frac m2}_{\infty,Q_{\rho}}.
\end{equation}

Now fix $\delta\in (0,1)$ and let $\sigma_0=\delta$, $\sigma_{i+1}=\sigma_i+(1-\sigma_i)/4$,
which satisfy $1-\sigma_i=(3/4)^i(1-\delta)$.
Applying \eqref{moserit} to $\sigma=\sigma_i$ and $\rho=\sigma_{i+1}$ for each $i$,
\[
\|u\|_{\infty,Q_{\sigma_i}} \leq \left(\tfrac 43\right)^{(1+\frac n2)i}
G(B)\cdot(1-\delta)^{-1-\frac n2}\,\|u\|^{1-\frac m2}_{\infty,Q_{\sigma_{i+1}}},
\]
where $G(B):=F(B)\cdot\|u\|^{\frac m2}_{m,Q}$. Therefore, for any $i$,
\[
\|u\|_{\infty,Q_{\delta}} \leq \left(\tfrac 43\right)^{(1+\frac n2)\Sigma j(1-\frac m2)^j}
\left[G(B)\cdot(1-\delta)^{-1-\frac n2}\right]^{\Sigma(1-\frac m2)^j}
\|u\|^{(1-\frac m2)^i}_{\infty,Q_{\sigma_i}},
\]
where $\Sigma$ denotes the summations from $0$ to $i-1$. Letting $i\to \infty$,
\[
\|u\|_{\infty,Q_{\delta}}\le\left(\tfrac 43\right)^{\frac{2-m}{m^2}(2+n)}
\left[G(B)\cdot(1-\delta)^{-1-\frac n2}\right]^{\frac 2m},
\]
that is,
\[
\|u\|_{\infty,Q_{\delta}}\le C(n,m)(1-\delta)^{-\frac{2+n}{m}} e^{-\frac{\mu}{m}}(c_0+1)^{\frac{n}{2m}}r^{-\frac{2+n}{m}}\|u\|_{m,Q}.
\]
Hence we get \eqref{Lm-mean} when $0<m<2$ .
\end{proof}

Theorem \ref{prop} gives a local $L^1$-mean
value inequality for harmonic functions on shrinkers.

\begin{corollary}\label{corelli}
Let $(M,g, f)$ be an $n$-dimensional shrinker satisfying \eqref{Eq1}
and \eqref{Eq2}. Assume there exists a nonnegative constant $c_0$ such that
\[
\mathrm{R}(x)\cdot r^2(x,o)\le c_0.
\]
For any nonnegative subharmonic function $u$ on $M$,
\[
\sup_{x\in B_p(\frac r2)}u(x)\le C(n)\frac{(c_0+1)^{\frac n2}}{e^{\mu}\,r^n}\int_{B_p(r)}u(x)dv(x)
\]
for any point $p\in \partial B_o(2r)$, where $\mu:=\mu(g,1)$ is the Perelman's entropy functional.
\end{corollary}

The scalar curvature condition in Theorem \ref{prop} is unnecessary if one considers
the Schr\"odinger heat equation. Indeed we can apply Lemma \ref{lem2} to get a
local mean value inequality for the Schr\"odinger heat equation on shrinkers without
any assumption.
\begin{theorem}\label{meaninequ2}
Let $(M,g, f)$ be an $n$-dimensional shrinker
satisfying \eqref{Eq1} and \eqref{Eq2}. Fix $0<m<\infty$. Then there exists a
positive constant $C(n,m)$ depending on $n$ and $m$, such that, for any
$s\in \mathbb{R}$, $0<\delta<1$, and for any nonnegative smooth solution $u$ of
\[
\left(\Delta-a\mathrm{R}-\partial_t\right)u(x,t)\ge0,\quad a>0,
\]
in the space-time cylinder $Q:=B_p(r)\times(s-r^2,s)$, where $p\in M$ and $r>0$, we have
\begin{equation}\label{Lm-meanin2}
\sup_{Q_\delta}\{u^m\}
\leq \frac{C(n,m)(a^{-1}+1)^{\frac n2}}{(1-\delta)^{2+n}\,e^{\mu}\,r^{2+n}}\int_Qu^m\,\, dv\, dt,
\end{equation}
where $Q_\delta:=B_p(\delta r)\times (s-\delta r^2,s)$
and $\mu:=\mu(g,1)$ is the Perelman's entropy functional.
\end{theorem}

\begin{proof}[Proof of Theorem \ref{meaninequ2}]
We only check the case $m=2$. The other cases are similar to the proof of Theorem \ref{prop}.
For $m\geq1$,
\begin{equation*}
\begin{aligned}
\left(\Delta-a\mathrm{R}-\partial_t\right) u^m&=mu^{m-1}\left(\Delta-a\mathrm{R}-\partial_t\right)u
+m(m-1)u^{m-2}|\nabla u|^2+(m-1)a\mathrm{R}u^m\\
&\ge mu^{m-1}\left(\Delta-a\mathrm{R}-\partial_t\right)u,
\end{aligned}
\end{equation*}
where we used $\mathrm{R}\ge 0$. This implies that if $u$ is a nonnegative
solution of $(\Delta-a\mathrm{R}-\partial_t)u\ge 0$, then $v:=u^m$ is also a
nonnegative solution of $(\Delta-a\mathrm{R}-\partial_t)v\ge 0$.

For any nonnegative function $\phi\in C^{\infty}_0(B)\times\mathbb{R}$,
where $B=B_p(r)$, multiplying $(\Delta-a\mathrm{R}-\partial_t)u\ge 0$
by $\phi$ and integrating it over $B$,
\[
\int_B(\phi u_t+\nabla\phi\nabla u+a\mathrm{R}\phi u)dv\le 0.
\]
Let $\phi=\psi^2u$, where $\psi\in C^{\infty}_0(B)$. Then
\begin{equation*}
\begin{aligned}
\int_B(\psi^2 uu_t+\psi^2|\nabla u|^2+a\mathrm{R}\psi^2u^2)dv
&\le 2\left|\int_B u\psi \nabla u\nabla\psi dv\right|\\
&\le 2\int_B|\nabla\psi|^2u^2dv+\frac 12\int_B\psi^2|\nabla u|^2dv,
\end{aligned}
\end{equation*}
which implies
\[
\int_B(2\psi^2 uu_t+|\nabla(\psi u)|^2+2a\mathrm{R}\psi^2u^2)dv\le 4\,\|\nabla\psi\|^2_{\infty} \int_{\mathrm{supp}(\psi)} u^2dv.
\]
Multiplying both sides of the above inequality by a time function $\lambda:=\lambda(t)$,
which will be determined later, and integrating by parts, we get
\begin{equation*}
\begin{aligned}
\partial_t\left(\int_B(\lambda\psi u)^2dv\right)+&\lambda^2\int_B\left(|\nabla(\psi u)|^2+2a\mathrm{R}\psi^2u^2\right)dv\\
&\leq C\lambda\Big(\lambda\|\nabla \psi\|^2_{\infty}+|\lambda'|\sup\psi^2\Big)\int_{\mathrm{supp}(\psi)} u^2dv.
\end{aligned}
\end{equation*}
Following the proof of Theorem \ref{prop}, we choose $\psi$ and $\lambda$ satisfying two properties:
\begin{enumerate}
\item
$0\leq\psi\leq 1$, $\mathrm{supp}(\psi)\subset\sigma B$, \, $\psi=1$ in $\sigma' B$ and
$|\nabla\psi|\leq \frac{2}{\kappa\,r}$;
\item
$0\leq\lambda\leq 1$, $\lambda=0$ in $(-\infty,s-\sigma r^2)$,  $\lambda=1$ in $(s-\sigma' r^2,+\infty)$, and
$|\lambda'(t)|\leq \frac{2}{\kappa^2 r^2}$,
where $0<\sigma'<\sigma<1$, $\kappa=\sigma-\sigma'$.
\end{enumerate}

Set $I_\sigma=(s-\sigma r^2,s)$. Integrating the above inequality over interval
$(s-r^2, t)$ with $t\in I_{\sigma'}$,
\begin{equation}
\begin{aligned}\label{integso2}
\sup_{I_{\sigma'}}\left\{\int_B\psi u^2dv\right\}+\int\int_{B\times I_{\sigma'}}\left(|\nabla(\psi u)|^2+2a\mathrm{R}\psi^2u^2\right)dvdt
\leq \frac{C}{\kappa^2 r^2}\int\int_{Q_\sigma}u^2dv dt.
\end{aligned}
\end{equation}

On the other hand, by the H\"older inequality and Lemma \ref{lem2},
for any $\varphi\in C^{\infty}_0(\widetilde{B})$,
where $\widetilde{B}\subset M$, we have
\[
\int_{\widetilde{B}}\varphi^{2(1+\frac 2n)}dv
\le \left(\int_{\widetilde{B}}\varphi^2dv\right)^{\frac{2}{n}}
\left[C(n)e^{-\frac{2\mu}{n}}\int_{\widetilde{B}}(4|\nabla \varphi|^2+\mathrm{R}\,\varphi^2)dv\right].
\]
Let $\varphi=u$ and $\widetilde{B}=B_p(\sigma' r)$
in above. Then we integrate it from $s-\sigma'r^2$ to $s$ with respect to
the time variable and get that
\begin{equation}
\begin{aligned}\label{space-timeineq2}
&\int^s_{s-\sigma'r^2}\int_{B_p(\sigma' r)}u^{2(1+\frac 2n)}dvdt\\
&\quad\le C(n)e^{-\frac{2\mu}{n}}\left(\int_{B_p(\sigma' r)}u^2dv\right)^{\frac{2}{n}}
\left[\int^s_{s-\sigma'r^2}\int_{B_p(\sigma' r)}(4|\nabla u|^2+\mathrm{R}\,u^2)dvdt\right].
\end{aligned}
\end{equation}
Here we will estimate each term of the right hand side of \eqref{space-timeineq2}
by using \eqref{integso2}. First, \eqref{integso2} implies
\[
\int_{B_p(\sigma' r)}u^2dv\le\int_{B_p(r)}\psi u^2dv
\le \frac{C}{\kappa^2 r^2}\int\int_{Q_\sigma} u^2dv dt.
\]
Second, we apply \eqref{integso2} to get
\[
\int^s_{s-\sigma'r^2}\int_{B_p(\sigma' r)}|\nabla u|^2dvdt
\le\int^s_{s-\sigma'r^2}\int_{B_p(r)}|\nabla\psi u|^2dvdt
\le \frac{C}{\kappa^2 r^2}\int\int_{Q_\sigma} u^2dv dt.
\]
Thirdly, we apply \eqref{integso2} to estimate that
\[
\int^s_{s-\sigma'r^2}\int_{B_p(\sigma' r)}\mathrm{R}\,u^2dvdt
\le\int^s_{s-\sigma'r^2}\int_{B_p(r)}\mathrm{R}\,(\psi u)^2dvdt
\le \frac{Ca^{-1}}{\kappa^2 r^2}\int\int_{Q_\sigma} u^2dv dt.
\]
Substituting the above three inequalities into \eqref{space-timeineq2} gives
\[
\int\int_{Q_{\sigma'}}u^{2(1+2/n)}dv dt
\leq C(n)e^{-\frac{2\mu}{n}}(a^{-1}+1)\left(\frac{C}{\kappa^2 r^2}\int\int_{Q_\sigma} u^2dv dt\right)^{1+2/n}.
\]
This inequality is nearly the same as \eqref{Moserinequ}, appeared
in the proof of Theorem \ref{prop}. Following the preceding argument, we apply the
same Moser's iteration to get \eqref{Lm-meanin2}.
\end{proof}

Theorem \ref{meaninequ2} gives a local $L^1$-mean value inequality
for a Schr\"odinger equation on shrinkers without any condition,
which is prepared for proving Theorem \ref{mainpoten}.

\begin{corollary}\label{corelli2}
Let $(M,g, f)$ be an $n$-dimensional shrinker satisfying \eqref{Eq1} and \eqref{Eq2}.
For any nonnegative solution $u$ of
\[
\left(\Delta-a\mathrm{R}\right)u(x)\ge0,\quad a>0,
\]
on $M$, we have
\[
\sup_{x\in B_p(\frac r2)}u(x)\le C(n)\frac{(a^{-1}+1)^{\frac n2}}{e^{\mu}\,r^n}\int_{B_p(r)}u(x)dv(x)
\]
for any point $p\in M$ and $r>0$, where $\mu:=\mu(g,1)$ is the Perelman's entropy functional.
\end{corollary}

\section{Proof of Theorem \ref{main1}}\label{sec4}

In this section we will follow the argument of \cite{[Li]} to prove Theorem \ref{main1}.
We first recall an important lemma due to Li \cite{[Li]}.

\begin{lemma}[Li \cite{[Li]}]\label{lem1}
Let $(M,g)$ be an $n$-dimensional manifold satisfying $V_p(r)\le c(n)r^\tau$
for some constant $\tau\ge 0$ at a base point $p\in M$ and let $K$
be a nonzero $k$-dimensional ($k<\infty$) subspace of $\mathcal{H}_d(M)$.
For $\beta>1$, $\delta>0$ and $r_0>0$, there exists $r_1>r_0$ such that
\[
\sum^k_{i=1}\int_{B_p(r_1)}u_i^2\, dv\ge k\cdot\beta^{-(2d+\tau+\delta)},
\]
where $\{u_i\}_{i=1}^k$ is an orthonormal basis of $K$ with respect to the
inner product $A_{\beta r_1}(u,v)=\int_{B_p(\beta r_1)} uv\, dv$.
\end{lemma}
\begin{remark}\label{schlowin}
The lemma is still true if the linear space $\mathcal{H}_d(M)$
is replaced by $\mathcal{H}_d(a, M)$.
\end{remark}

Using Li's proof strategy \cite{[Li]}, we give a explicit estimate
for the dimension of $\mathcal{H}_d(M)$ by combining Corollary
\ref{corelli} and Lemma \ref{lem1}.

\begin{proof}[Proof of Theorem \ref{main1}]
Let $K$ be any nonzero $k$-dimensional
linear subspace $K\subset \mathcal{H}_d(M)$ with a infimum point $p_0\in M$
of $f$. Since $\mathrm{dim}K$ and $\mathrm{dim}\mathcal{H}_d(M)$ both do not
depend on a base point, without loss of generality, we may choose a infimum
$p_0$ as their base point. To prove the theorem, it suffices to estimate $k$.
By Proposition \ref{volesti2}, $V_{p_0}(r)\le C(n)r^n$ for all $r>0$. Using
this, by Lemma \ref{lem1}, for any $\delta>0$, there exists $r_1>8r(o,p_0)$
such that
\begin{equation}\label{esti}
\sum^k_{i=1}\int_{B_{p_0}(r_1)}u_i^2(x)dv(x)\,\ge k \cdot2^{-(2d+n+\delta)},
\end{equation}
where $\{u_i\}^k_{i=1}$ is an orthonormal basis of $K$ with respect to the
inner product $A_{2r_1}(u,v)=\int_{B_{p_0}(2r_1)} uv\, dv$.

We shall apply \eqref{esti} to estimate $k$ when $d\ge0$. Since $\sum^k_{i=1}u^2_i(x)$
is subharmonic, by the maximum principle, there exists a point
$q\in\partial B_{p_0}(r_1)$ such that $\sum^k_{i=1}u^2_i(x)\le\sum^k_{i=1}u^2_i(q)$
for all $x\in B_{p_0}(r_1)$. Now we can find a $k\times k$ orthogonal matrix
$(a_{ij})$ such that functions $v_i(x)=\sum^k_{j=1}a_{ij}u_j$ satisfy
$v^2_1(q)=\sum^k_{i=1}u^2_i(q)$ and $v_j(q)=0$ for $2\le j\le k$. So
$\sum^k_{i=1}u^2_i(x)\le v^2_1(q)$. Integrating this over the ball
$B_{p_0}(r_1)$,
\begin{equation}\label{jifen}
\sum^k_{i=1}\int_{B_{p_0}(r_1)}u^2_i(x)\,dv(x)\le V_{p_0}(r_1)v^2_1(q).
\end{equation}
On the other hand, since $\mathrm{R}(x)\le \frac{c_0}{r^2(x,o)}$, and
$r(x,o)\ge r(p_0,x)-r(o,p_0)\ge\tfrac{1}{2}r(x,p_0)$ when
$r(p_0,x)\ge 2r(o,p_0)$, then we have a fact that
\[
\mathrm{R}(x)\le \frac{4c_0}{r^2(x,p_0)}\quad\mathrm{whenever}\quad r(p_0,x)\ge2r(o,p_0).
\]
Consider the nonnegative subharmonic function $v^2_1(x)$ in $B_q(r_1/2)$
with $\mathrm{R}(x)\le \frac{4c_0}{r^2(x,p_0)}$. Using Corollary
\ref{corelli}, we immediately have
\[
v^2_1(q)\le C(n)\frac{(c_0+1)^{\frac n2}}{e^{\mu}\,{r_1}^n}\int_{B_q(r_1/4)}v^2_1(y)dv(y).
\]
Using $B_q(r_1/4)\subset B_{p_0}(2r_1)$, we further have
\begin{equation*}
\begin{aligned}
V_{p_0}(r_1)v^2_1(q)&\le V_{p_0}(r_1)C(n)\frac{(c_0+1)^{\frac n2}}{e^{\mu}\,{r_1}^n}\int_{B_{p_0}(2r_1)}v^2_1(y)dv(y)\\
&\le C(n)(c_0+1)^{\frac n2}e^{-\mu}\int_{B_{p_0}(2r_1)}v^2_1(y)dv(y),
\end{aligned}
\end{equation*}
where we used $V_{p_0}(r_1)\le C(n)r^n_1$. Combining this with \eqref{esti} 
and \eqref{jifen}, we have
\[
k \cdot 2^{-(2d+n+\delta)}\le C(n)(c_0+1)^{\frac n2}e^{-\mu},
\]
since $\int_{B_{p_0}(2r_1)}v^2_1(y)dv(y)=1$.
Noticing that $\delta>0$ is arbitrary, hence
\[
k \le C(n)(c_0+1)^{\frac n2}e^{-\mu}\,4^d.
\]
Since the subspace $K$ is arbitrary, Theorem \ref{main1} follows.
\end{proof}

\section{Proof of Theorem \ref{mainpoten}}\label{sec5}
In this section, we will prove Theorem \ref{mainpoten} by using the
similar argument of Section \ref{sec4}. Since the second estimate of
Theorem \ref{mainpoten} is sharper, we also need the following important
lemma, whose novelty part is that there are no curvature assumption.

\begin{lemma}\label{parallemm2}
Let $(M,g, f)$ be an $n$-dimensional complete non-compact shrinker satisfying
\eqref{Eq1} and \eqref{Eq2} with a infimum point $p_0\in M$ of $f$. Let
$\{u_i\}_{i=1}^k$ be any basis of a nonzero $k$-dimensional subspace
$K\subset\mathcal{H}_d(a, M)$. For any $p\in M$, $r>>r(p,p_0)$ and any
$0<\epsilon<\tfrac 12$,
\[
\sum^k_{i=1}\int_{B_p(r)}u^2_i dv\le C(n)(a^{-1}+1)^{\frac n2}e^{-\mu}\,\epsilon^{1-n}\sup_{u\in\{\langle A,U\rangle\}}\int_{B_p((1+\epsilon)r)}u^2 dv,
\]
where $\langle A,U\rangle:=\{\sum_ia_iu_i{|}\sum_ia^2_i=1\}$ for some
unit vector $A=(a_1, \ldots, a_k)\in \mathbb{R}^k$ with $U=(u_1, \ldots, u_k)$,
and $\mu:=\mu(g,1)$ is the Perelman's entropy functional.
\end{lemma}

\begin{proof}[Proof of Lemma \ref{parallemm2}]
We will apply Li's proof trick \cite{[Li]}. For any $x\in B_p(r)$, let
$K_x=\{u\in K|u(x)=0\}$,
then it is at most codimension one of $K$. By an orthonormal change of basis,
we assume $u_i\in K_x$, $2\le i\le k$ and $\sum^k_{i=1}u^2_i(x)=u^2_1(x)$.
Since $(\Delta-a\mathrm{R})u^2_1=a\mathrm{R}u_1^2+2|\nabla u_1|^2\ge 0$,
by Corollary \ref{corelli2}, for any $0<\epsilon<\frac 12$, we have
\begin{equation*}
\begin{aligned}
\sum^k_{i=1}u^2_i(x)=u^2_1(x)
&\le\frac{C(n)(a^{-1}+1)^{\frac n2}}{e^{\mu}\,\big[(1+\epsilon)r-\rho(x)\big]^n}\int_{B_x((1+\epsilon)r-\rho(x))}u^2_1(y)dv(y)\\
&\le\frac{C(n)(a^{-1}+1)^{\frac n2}\, n}{e^{\mu}\,r^n\,\big[1+\epsilon-r^{-1}\rho(x)\big]^n}
\sup_{u\in\{\langle A,U\rangle\}}\int_{B_p((1+\epsilon)r)}u^2 dv,
\end{aligned}
\end{equation*}
for any $r>0$, where $\rho(x)$ is the distance function from $p$ to $x$.
Integrating over $B_p(r)$ yields
\begin{equation}
\begin{aligned}\label{keyimpor2}
\sum^k_{i=1}\int_{B_p(r)}u^2_i dv
&\le\frac{C(n)(a^{-1}+1)^{\frac n2}}{e^{\mu}\,r^n}
\sup_{u\in\{\langle A,U\rangle\}}\int_{B_p((1+\epsilon)r)}u^2 dv\\
&\quad\times \int_{B_p(r)}\big[1+\epsilon-r^{-1}\rho(x)\big]^{-n}dv(x).
\end{aligned}
\end{equation}
If we define $h(t):=(1+\epsilon-r^{-1}t)^{-n}$, then
$h'(t)=nr^{-1}(1+\epsilon-r^{-1}t)^{-n-1}\ge 0$ and
\begin{equation}\label{identity}
\int_{B_p(r)}\big[1+\epsilon-r^{-1}\rho(x)\big]^{-n}dv(x)=\int^r_0A_p(t)h(t)dt,
\end{equation}
where $A_p(t):=\mathrm{Area}(\partial B_p(t))$ satisfies $V_p'(t)=A_p(t)$ almost
everywhere. Now we will apply the property of $h(t)$ and Remark \ref{areacomp} to
give an upper estimate for $\int^r_0A_p(t)h(t)dt$.
By Remark \ref{areacomp}, $A_p(t)\le C(n)t^{n-1}$ for $t\ge r(p,p_0)$, and hence
\begin{equation*}
\begin{aligned}
\int^r_0A_p(t)h(t)dt&=\int^{r(p,p_0)}_0A_p(t)h(t)dt+\int^r_{r(p,p_0)}A_p(t)h(t)dt\\
&\le C(r(p,p_0))+C(n)\int^r_{r(p,p_0)}t^{n-1}h(t)dt\\
&\le C(r(p,p_0))+C(n)r^{n-1}\int^r_0h(t)dt\\
&= C(r(p,p_0))+C(n)r^{n-1}\int^r_0(1+\epsilon-r^{-1}t)^{-n}dt\\
&= C(r(p,p_0))+\frac{C(n)}{n-1}r^n\left[\epsilon^{1-n}-(1+\epsilon)^{1-n}\right]\\
&\le C(n)\,r^n\epsilon^{1-n}
\end{aligned}
\end{equation*}
for $r>>r(p,p_0)$, where $C(r(p,p_0))$ is a constant depending only on $r(p,p_0)$.
Combining this with \eqref{identity} and
\eqref{keyimpor2} proves the lemma.
\end{proof}

Now we apply Lemma \ref{parallemm2}, Remark \ref{schlowin} and Corollary
\ref{corelli2} to give a explicit estimate for the dimension of $\mathcal{H}_d(a,M)$.

\begin{proof}[Proof of Theorem \ref{mainpoten}]
As in Li's proof strategy \cite{[Li]}, for any point $p\in M$ and $\beta>1$,
let $K$ be any nonzero $k$-dimensional subspace $K\subset \mathcal{H}_d(a,M)$,
and $\{u_i\}^k_{i=1}$ be an orthonormal basis of $K$ with respect to the inner
product $A_{\beta r_1}(u,v)=\int_{B_p(\beta r_1)} uv\, dv$. Now we want to
estimate $k$. By Proposition \ref{volesti2}, $V_p(r)\le C(n) r^n$ for all
$r\ge r(p,p_0)$. Using this to Remark \ref{schlowin}, for any $\delta>0$,
letting $r_0=r(p,p_0)$, there exists $r_1>r_0$ such that
\begin{equation}\label{estip}
\sum^k_{i=1}\int_{B_p(r_1)}u_i^2\, dv\ge k \cdot\beta^{-(2d+n+\delta)}.
\end{equation}
We first apply \eqref{estip} to give a rough upper estimate for $k$ when $d\geq 0$.
Since $(\Delta-a\mathrm{R})\sum^k_{i=1}u^2_i(x)\ge 0$, then $\sum^k_{i=1}u^2_i(x)$
is subharmonic, by the maximum principle, there exists a point $q\in \partial B_p(r_1)$
such that $\sum^k_{i=1}u^2_i(x)\le \sum^k_{i=1}u^2_i(q)$ for all $x\in B_p(r_1)$.
Now we can find a $k\times k$ orthogonal matrix $(a_{ij})$ such
that $v_i(x)=\sum^k_{j=1}a_{ij}u_j$ with $v^2_1(q)=\sum^k_{i=1}u^2_i(q)$
and $v_j(q)=0$ for $2\le j\le k$. So $\sum^k_{i=1}u^2_i(x)\le v^2_1(q)$.
Integrating this yields
\begin{equation}\label{jifengg}
\sum^k_{i=1}\int_{B_p(r_1)}u^2_i(x)\,dv(x)\le V_p(r_1)v^2_1(q).
\end{equation}
On the other hand, applying Corollary \ref{corelli2} to $v^2_1$ and noticing
$B_q(r_1)\subset B_p(2r_1)$, we get
\begin{equation*}
\begin{aligned}
V_p(r_1)v^2_1(q)&\le V_p(r_1)C(n)\frac{(c_0+1)^{\frac n2}}{e^{\mu}\,r_1^n}\int_{B_q(r_1)}v^2_1(y)dv(y)\\
&\le V_p(r_1)C(n)\frac{(c_0+1)^{\frac n2}}{e^{\mu}\,r_1^n}\int_{B_p(2r_1)}v^2_1(y)dv(y)\\
&\le C(n)(c_0+1)^{\frac n2}e^{-\mu}\int_{B_p(2r_1)}v^2_1(y)dv(y),
\end{aligned}
\end{equation*}
where we used $V_p(r_1)\le C(n) r_1^n$ in the last inequality. Combining this with
\eqref{jifengg} and \eqref{estip} for $\beta=2$, since $\int_{B_p(2r_1)}v^2_1(y)dv(y)=1$,
we get
\[
k \cdot 2^{-(2d+n+\delta)}\le C(n)(c_0+1)^{\frac n2}e^{-\mu}.
\]
Since $\delta>0$ is arbitrary and subspace $K$ is arbitrary, then
$k \le C(n)(c_0+1)^{\frac n2}e^{-\mu}\,4^d$ and first estimate of Theorem
\ref{mainpoten} follows.

Next we prove the second estimate of Theorem \ref{mainpoten}.
Letting $\beta=1+\epsilon$ in Lemma \ref{parallemm2},
\[
\sum^k_{i=1}\int_{B_p(r)}u^2_i\,dv\le C(n)(c_0+1)^{\frac n2}e^{-\mu}\,\epsilon^{1-n}
\]
for all $r>>r(p,p_0)$, because $\int_{B_p((1+\epsilon)r)}u^2 dv=1$ for
$u\in\{\langle A,U\rangle\}$, where $\{u_i\}^k_{i=1}$ is an orthonormal basis
of $K\subset \mathcal{H}_d(a, M)$ with respect to the inner product
$A_{\beta r}(u,v)=\int_{B_p(\beta r)} uv\, dv$. Combining the above inequality
and \eqref{estip} with a possibly different $r_1$, for $d\geq 1$, setting
$\epsilon=\tfrac{1}{2d}$ and letting $\delta\rightarrow 0$, we get
\begin{equation*}
\begin{aligned}
k&\le C(n)\left(1+\tfrac{1}{2d}\right)^{(2d+n+\delta)}\left(\tfrac{1}{2d}\right)^{(1-n)}(c_0+1)^{\frac n2}e^{-\mu}\\
&\leq C(n)(c_0+1)^{\frac n2}e^{-\mu}\,d^{n-1},
\end{aligned}
\end{equation*}
because $(1+\tfrac{1}{2d})^{(2d+n+\delta)}$ is bounded above. Since the
$k$-dimensional subspace $K$ is arbitrary, the second estimate of
Theorem \ref{mainpoten} follows.
\end{proof}

\section{Polynomial growth ancient caloric functions}\label{sec6}

A natural generalization of the ($f$-)harmonic function is a ancient solution,
defined on all negative time, with polynomial growth of the ($f$-)heat equation. In this section,
we will generalize the preceding results to the caloric function setting.

For an $n$-dimensional Riemannian manifold $(M,g)$ and a smooth function
$f$ on $M$, a space-time function $u(x,t)$ is called $f$-caloric function if it satisfies
the $f$-heat equation
\[
(\Delta_f-\partial_t)u=0.
\]
For a fixed constant $d\ge0$, we denote by $\mathcal{P}^f_d(M)$ the linear
space of all ancient $f$-caloric functions with polynomial growth of degree at most $d$
satisfying that there exist some point $p\in M$ and a constant $C(u)$ depending on $u$,
\[
\sup_{B_p(r)\times [-r^2,0]}|u|\le C(u)(1+r)^d
\]
for sufficiently large $r$. When $f$ is constant, $\mathcal{P}^f_d(M)$ is simply written
as $\mathcal{P}_d(M)$, and it is a linear space of ancient caloric functions
with polynomial growth of degree at most $d$.

In \cite{[Ca1],[Ca2]}, Calle initiated the study of dimension bounds for
$\mathcal{P}_d(M)$, which plays an important role in understanding the higher
codimension mean curvature flow \cite{[CoMi3]}. When  $(M,g)$ has nonnegative
Ricci curvature, Souplet and Zhang used the gradient estimate technique to
prove that $\mathrm{dim} \mathcal{P}_d(M)=1$ for all $d<1$ (see Theorem 1.2 (b)
in \cite{[SZ]}). For each $d\geq1$, Lin and Zhang \cite{[LiZh]} proved that
$\mathrm{dim} \mathcal{P}_d(M)\le C(n)d^{n+1}$. Based on a Lin-Zhang's
observation that ancient caloric functions of polynomial growth are
polynomials in time (Theorem 1.2 (b) in \cite{[LiZh]}), Colding and
Minicozzi \cite{[CoMi2]} improved Lin-Zhang's estimate and got that
\[
\mathrm{dim} \mathcal{P}_d(M)\le C(n)d^n
\]
for $d\geq1$. The power $n$ is sharp because
$\mathrm{dim}\mathcal{P}_d(\mathbb{R}^n)\sim C(n)d^n$ as $d\to \infty$.
Recently, Colding and Minicozzi \cite{[CoMi3]} developed some new technique
and gave a sharp bound for the codimension of an ancient mean curvature flow
by the entropy.

In this section, we first prove a sharp dimension estimate of the space
$\mathcal{P}_d(M)$ on a shrinker when the scalar curvature is of at least
quadratic decay.

\begin{theorem}\label{Main1c}
Let $(M,g, f)$ be an $n$-dimensional complete non-compact shrinker satisfying
\eqref{Eq1} and \eqref{Eq2}. If there exists a nonnegative constant $c_0$
such that
\[
\mathrm{R}(x)\cdot r^2(x,o)\le c_0,
\]
then the dimension of $\mathcal{P}_d(M)$ is finite for each
$d\geq 1$. Indeed, for each $d\ge 1$,
\[
\mathrm{dim}\mathcal{P}_d(M)\le C(n)(c_0+1)^{\frac n2}e^{-\mu}\,d\cdot4^d,
\]
where $\mu$ is the Perelman's entropy functional.
\end{theorem}

\begin{proof}[Proof of Theorem \ref{Main1c}]
We apply Theorem \ref{main1} to prove Theorem \ref{Main1c}. By Proposition
\ref{volesti2}, the volume of geodesic ball is at most Euclidian volume growth.
Using this property and a result of Colding and Minicozzi (see Theorem 0.3
in \cite{[CoMi2]}), we immediately get
\[
\mathrm{dim}\mathcal{P}_{2m}(M)\le (m+1)\,\mathrm{dim}\mathcal{H}_{2m}(M)
\]
for all $m\geq1$. Combining this with Theorem \ref{main1} completes the proof.
\end{proof}

We also prove a dimension estimate of $\mathcal{P}^f_d(M)$
on shrinkers with out any assumption.

\begin{theorem}\label{main2}
Let $(M,g, f)$ be an $n$-dimensional complete non-compact shrinker satisfying
\eqref{Eq1} and \eqref{Eq2}. Then, for each $d\ge 1$,
\[
\mathrm{dim}\mathcal{P}^f_d(M)\leq C(n)d
\]
and for each $0\le d<1$,
\[
\mathrm{dim}\mathcal{P}^f_d(M)=1.
\]
\end{theorem}

The proof of first estimate is the same as the case of Theorem \ref{Main1c}, and
we need to use Theorem \ref{main0}. The proof of second part is a little
complicated, where a key ingredient is an elliptic type gradient
estimate, which seems to be of independent interest in a shrinker.
\begin{proposition}\label{prpest}
Let $(M,g, f)$ be an $n$-dimensional complete non-compact shrinker
satisfying \eqref{Eq1}  and \eqref{Eq2}. Fix any fixed point $p\in M$ and a number $R\ge2\sigma_0$,
where constant $\sigma_0:=6\sqrt{f(p)}+8n-\frac 83$ (which is determined in Proposition
\ref{meancom}). If $0<u(x,t)\le D$ for some constant $D$, is a smooth solution to $f$-heat equation
in $Q_{R,T}:=B_p(R)\times[t_0-T,t_0]\subset M\times(-\infty,\infty)$,
where $t_0\in \mathbb{R}$ and $T>0$, then
\begin{equation}\label{heor1}
|\nabla\ln u|\le C(n)\left(\frac 1R+\frac{1}{\sqrt{t-t_0+T}}\right)
\left(1+\ln \frac Du\right)
\end{equation}
in $Q_{R/2, T}$ with $t\neq t_0-T$.
\end{proposition}
\begin{proof}[Proof of Proposition \ref{prpest}]
The proof strategy comes from \cite{[Wu]}.
Since \eqref{heor1} is invariant under the scaling $u\to u/D$,
without loss of generality, we assume $0<u\leq 1$. Define
$h(x,t):=\ln u(x,t)$. Then $h(x,t)$ is non-positive and satisfies
$(\Delta_f-\partial_t)h+|\nabla h|^2=0$. Set
\[
\omega:=|\nabla\ln(1-h)|^2
\]
Then $\omega$ in $Q_{R,T}$ satisfies
\begin{equation*}
\begin{aligned}
(\Delta_f-\partial_t)\omega&=\frac{2h}{1-h}\langle \nabla h,\nabla\omega\rangle+2(1-h)\omega^2
+2\frac{\mathrm{Ric}_f(\nabla h,\nabla h)}{(1-h)^2}\\
&\quad+2\left(\frac{|\mathrm{Hess}\,h|^2}{(1-h)^2}
+2\frac{\mathrm{Hess}\,h(\nabla h, \nabla h)}{(1-h)^3}+\frac{|\nabla h|^4}{(1-h)^4}\right).
\end{aligned}
\end{equation*}
Since $\mathrm{Ric}_f\ge 0$ and the last term of above equality is nonnegative, 
$\omega$ in $Q_{R,T}$ further satisfies
\begin{equation}\label{lemmaequ3}
(\Delta_f-\partial_t)\omega\geq\frac{2h}{1-h}\langle \nabla h,\nabla\omega\rangle+2(1-h)\omega^2.
\end{equation}

In the following, we will apply the maximum principle in $Q_{R,T}$ to prove 
\eqref{heor1}. To achieve it, we need  a space-time cut-off function. This 
is also appeared in \cite{[Li-Yau]}, \cite{[SZ]} and \cite{[Wu]}.

\begin{lemma}\label{cutoff}
Fix $t_0\in \mathbb{R}$ and $T>0$. For any $\tau\in(t_0-T,t_0]$, there exists a smooth function
${\overline{\psi}}:[0,\infty)\times[t_0-T,t_0]\to\mathbb R$ satisfying the following four propositions.
\begin{enumerate}
\item $0\le {\overline{\psi}}(r,t)\le 1$ in the space-time set $[0,R]\times[t_0-T,t_0]$,
where the real number $R>0$, which is supported in a open subset of $[0,R]\times[t_0-T,t_0]$.
\item ${\overline{\psi}}(r,t)=1$ and $\partial_r{\overline{\psi}}(r,t)=0$
in $[0,R/2]\times[\tau,t_0]$ and $[0,R/2]\times[t_0-T,t_0]$, respectively.
\item $|\partial_t{\overline{\psi}}|\le\frac{C{\overline{\psi}}^{\frac12}}{\tau-(t_0-T)}$
in $[0,\infty)\times[t_0-T,t_0]$ for some constant $C>0$, and ${\overline{\psi}}(r,t_0-T)=0$
for all $r\in[0,\infty)$.
\item $-\frac{C_\epsilon{\overline{\psi}}^\epsilon}{R}\le \partial_r{\overline{\psi}} \le 0$
and $\left|\partial_r^2{\overline{\psi}}\right|\leq\frac{C_\epsilon{\overline{\psi}}^\epsilon}{R^2}$
in $[0,\infty)\times[t_0-T,t_0]$ for every $0<\epsilon<1$ with some constant
$C_\epsilon$ depending only on $\epsilon$.
\end{enumerate}
\end{lemma}

The outline of proof is as follows. Pick any $\tau\in(t_0-T,t_0]$ and fix
a cutoff function ${\overline{\psi}}(r,t)$ as in Lemma \ref{cutoff}.
We can show that \eqref{heor1} holds at the space-time
point $(x,\tau)$ for all $x$ such that $d(x,p)<R/2$. Since
$\tau\in(t_0-T,t_0]$ is arbitrary, the conclusion follows.

Now we give a detailed discussion. Consider a cutoff function
$\psi:M\times[t_0-T,t_0]\to \mathbb R$, such that
$\psi={\overline{\psi}}(d(x,p),t)\equiv\psi(r,t)$.
Then $\psi(x,t)$ is be a smooth cut-off function, supported in $Q_{R,T}$.
Using Lemma \ref{cutoff}, we carefully estimate each term in the evolution equation
$(\Delta_f-\partial_t)(\psi\omega)$ at a space-time point where $\psi\omega$ attains
its maximum. By \eqref{lemmaequ3}, we get
\begin{equation}
\begin{aligned}\label{lemdx3}
\left(\Delta_f-\partial_t\right)(\psi\omega)&\ge2\left(\frac{h}{1-h}\nabla
h+\frac{\nabla\psi}{\psi}\right)\cdot\nabla(\psi\omega)+2\psi(1-h)\omega^2\\
&\quad-2\left(\frac{h}{1-h}\nabla h\cdot\nabla\psi+\frac{|\nabla\psi|^2}{\psi}\right)\omega
+(\Delta_f-\partial_t)\psi\cdot\omega.
\end{aligned}
\end{equation}
Let $(x_1,t_1)$ be a maximum space-time point for the function
$\psi\omega$ in the following closed set
\[
\left\{(x,t)\in M\times[t_0-T,\tau]\,|d(x,p)\leq R\right\}.
\]
Assume $(\psi\omega)(x_1,t_1)>0$; otherwise, $\omega(x,\tau)\leq0$ and
\eqref{heor1} naturally holds at $(x,\tau)$ whenever $d(x, p)<R/2$.
Here $t_1\neq t_0-T$, since we assume $(\psi\omega)(x_1,t_1)>0$. By the
standard Calabi's argument \cite{[Cala]}, we can also assume that
$\psi(x,t)$ is smooth at $(x_1,t_1)$ due to. Since $(x_1,t_1)$ is a maximum
space-time point, at $(x_1,t_1)$,
\[
\Delta_f(\psi\omega)\leq0,\quad(\psi\omega)_t\geq0
\quad \mathrm{and}\quad\nabla(\psi\omega)=0.
\]
Using these estimates, at $(x_1,t_1)$, \eqref{lemdx3} can be simplified as
\begin{equation}\label{lefor}
\begin{aligned}
2\psi(1-h)\omega^2\leq\left(\frac{2h}{1-h}\nabla h\cdot\nabla\psi
+2\frac{|\nabla\psi|^2}{\psi}\right)\omega-(\Delta_f\psi)\omega+\psi_t\omega.
\end{aligned}
\end{equation}

The rest of this part we will apply \eqref{lefor} to give \eqref{heor1}.
We shall divide two case according the point $x_1$ whether lies in
$B_p(\sigma_0)$ or not.

If $x_1\in B_p(\sigma_0)$, where $\sigma_0:=6\sqrt{f(p)}+8n-\frac 83$ is
defined in Proposition \ref{meancom}, then $\psi$ is constant in space
direction in $B_p(R/2)$ by our assumption, where $R\geq2\sigma_0$. So
\eqref{lefor} indeed yields the estimate
\[
\omega\le \frac{\psi_t}{2\psi}\leq\frac{C}{\tau-(t_0-T)}
\]
at $(x_1,t_1)$, where we used $1-h\geq 1$ and Lemma \ref{cutoff} (3).
Since $\psi(x,\tau)=1$ when $d(x,p)<R/2$ by the proposition
(2) in Lemma \ref{cutoff}, the above estimate gives that
\[
\omega(x,\tau)=(\psi\omega)(x,\tau)
\le(\psi\omega)(x_1,t_1)
\le\omega(x_1,t_1)
\le\frac{C}{\tau-t_0+T}
\]
for all $x\in M$ such that $d(x,p)<R/2$. By the definition of $w(x,\tau)$
and a fact that $\tau\in(t_0-T,t_0]$ was chosen arbitrarily, we in fact
prove that
\[
\frac{|\nabla
h|}{(1-h)}(x,t)\leq\frac{C}{\sqrt{t-t_0+T}}
\]
for all $(x,t)\in Q_{R/2,T}:=B_p(R/2)\times[t_0-T,t_0]$ with
$t\neq t_0-T$. Then \eqref{heor1} follows since $h=\ln(u/D)$ with
$D$ scaled to $1$.

Thus we consider the case: $x_1\not\in B_p(\sigma_0)$. Since $\text{Ric}_f=\frac 12g$
and $d(x_1,p)\geq \sigma_0$ in $B_p(R)$, by Proposition \ref{meancom},
we have the weighted Laplacian comparison
\begin{equation}\label{gencomp}
\Delta_f\,r(x_1)\le 0.
\end{equation}
This comparison result will be used later. Below we will carefully estimate upper
bounds for each term of the right-hand side of \eqref{lefor}. Notice that
the Young's inequality
\[
ab\le {a^m}/{m}+{b^n}/{n},\quad \forall\,\,\, m,n>0
\,\,\,\mathrm{with}\,\,\, 1/m+1/n=1
\]
will be repeatedly used in the following. Meanwhile we let $c$ denote a
constant depending only on $n$ whose value may change from line to line.

First, for the first term of the right hand side of \eqref{lefor}, we have
the estimates
\begin{equation}
\begin{aligned}\label{term1}
\left(\frac{2h}{1-h}\nabla h\cdot\nabla\psi\right)\omega
&\leq2|h|\cdot|\nabla\psi|\cdot\omega^{3/2}\\
&\leq\psi(1-h)\omega^2+c
\frac{(h|\nabla\psi|)^4}{[\psi(1-h)]^3}\\
&\leq\psi(1-h)\omega^2+c\frac{h^4}{R^4(1-h)^3}.
\end{aligned}
\end{equation}

For the second term of the right hand side of (\ref{lefor}), we have
\begin{equation}
\begin{aligned}\label{term2}
2\frac{|\nabla\psi|^2}{\psi}\omega
&\leq\frac 16\psi\omega^2+c
\left(\frac{|\nabla\psi|^2}{\psi^{3/2}}\right)^2\\
&\leq\frac 16\psi\omega^2+\frac{c}{R^4}.
\end{aligned}
\end{equation}

For the third term of the right hand side of \eqref{lefor}, since $\psi$
is a radial function, then at $(x_1,t_1)$, using \eqref{gencomp} we can
estimate it as follows
\begin{equation}
\begin{aligned}\label{term3}
-(\Delta_f\psi)\omega&=-\left[(\partial_r\psi)\Delta_fr+(\partial^2_r\psi)\cdot
|\nabla r|^2\right]\omega\\
&\leq(-\partial^2_r\psi) \omega\\
&\leq\frac 16\psi\omega^2+c\left(\frac{|\partial^2_r\psi|}{\psi^{1/2}}\right)^2\\
&\leq\frac 16\psi\omega^2+\frac{c}{R^4},
\end{aligned}
\end{equation}
where in the last inequlity we used Lemma \ref{cutoff} (4).

Finally, we estimate the last term:
\begin{equation}
\begin{aligned}\label{term4}
|\psi_t|\omega&\leq\frac 16\left(\psi^{1/2}\omega\right)^2+c
\left(\frac{|\psi_t|}{\psi^{1/2}}\right)^2\\
&\leq\frac 16\psi\omega^2+\frac{c}{(\tau-t_0+T)^2}.
\end{aligned}
\end{equation}

\

Substituting \eqref{term1}-\eqref{term4} into the right hand side of \eqref{lefor},
we have
\[
\psi(1-h)\omega^2\le\frac{ch^4}{R^4(1-h)^3}
+\frac 12\psi\omega^2+\frac{c}{R^4}+\frac{c}{(\tau-t_0+T)^2}
\]
at $(x_1,t_1)$. Since $1-h\geq1$, then
\[
\psi\omega^2\le\frac{ch^4}{R^4(1-h)^4}+\frac 12\psi\omega^2
+\frac{c}{R^4}+\frac{c}{(\tau-t_0+T)^2}
\]
at $(x_1,t_1)$. Moreover, since $h^4\le (1-h)^4$, the above
inequality implies
\[
(\psi^2\omega^2)(x_1,t_1)\le(\psi\omega^2)(x_1,t_1)
\le\frac{c}{R^4}+\frac{c}{(\tau-t_0+T)^2}.
\]
Since $\psi(x,\tau)=1$ when $d(x,p)<R/2$ by Lemma \ref{cutoff} (2),
from the above estimate, we get
\[
\omega(x,\tau)=(\psi\omega)(x,\tau)
\le(\psi\omega)(x_1,t_1)
\le\frac{c}{R^2}+\frac{c}{\tau-t_0+T}
\]
for all $x\in M$ such that $d(x,p)<R/2$. By the definition of
$w(x,\tau)$ and the fact that $\tau\in(t_0-T,t_0]$ was chosen
arbitrarily, we in fact show that
\[
\frac{|\nabla h|}{(1-h)}(x,t)\leq\frac cR+\frac{c}{\sqrt{t-t_0+T}}
\]
for all $(x,t)\in Q_{R/2,T}:=B_p(R/2)\times[t_0-T,t_0]$ with
$t\neq t_0-T$. This implies the conclusion since $h=\ln(u/D)$
with $D$ scaled to $1$ and $R\geq2\sigma_0$.
\end{proof}

Now we complete the proof of Theorem \ref{main2}.
\begin{proof}[Proof of Theorem \ref{main2}]
On a non-compact shrinker $(M,g, f)$, for any $p\in M$ and $r>0$,
\[
V^f_p(r):=\int_{B_p(r)}e^{-f}dv\le\int_Me^{-f}dv<+\infty,
\]
where we used Corollary 1.1 in \cite{[CaZh]}. Using this volume estimate
and following the argument of Colding and Minicozzi \cite{[CoMi2]}, we can prove that
\[
\mathrm{dim}\mathcal{P}^f_{2m}(M)\le (m+1)\,\mathrm{dim}\mathcal{H}^f_{2m}(M)
\]
for all $m\geq1$. Combining this with a fact that $\mathrm{dim}\mathcal{H}^f_{2m}(M)=1$
in Theorem \ref{main0} yields the first part of Theorem \ref{main2}.

Then we prove the second part of Theorem \ref{main2}. For any fixed space-time point
$(x_0, t_0)\in M\times(-\infty,0]$, we let $D_R:=\sup_{(x,t)\in Q_{R,R^2}}|u(x,t)|$,
where $Q_{R,R^2}:=B_{x_0}(R)\times[t_0-R^2,t_0]$. Consider a smooth function
$U(x,t):=u+2D_{2R}$ and it satisfies $D_{2R}\le U(x,t)\le 3D_{2R}$ whenever
$(x,t)\in Q_{2R,4R^2}$. Then, we apply Proposition \ref{prpest} to $U(x,t)$
and obtain
\[
\frac{|\nabla u(x_0, t_0)|}{u(x_0, t_0)+2D_{2R}}
\leq\frac{C(n)}{R}
\]
for $R\geq2\sigma_0$, where $\sigma_0:=6\sqrt{f(x_0)}+8n-\frac 83$. Since
$f$-caloric function $u(x,t)$ is ancient and satisfies the sublinear growth,
by letting $R\to\infty$, then $D_{2R}=o(R)$ and the above estimate gives
$\nabla u(x_0, t_0)=0$. This implies that $u(x,t)$ is constant since
$(x_0, t_0)$ is arbitrary.
\end{proof}


\end{document}